\documentclass[a4paper,12pt]{article}

\usepackage[T2A]{fontenc}
\usepackage[utf8]{inputenc}
\usepackage[russian]{babel}
\usepackage{amsmath, amssymb, amsthm}
\usepackage{enumerate}
\usepackage{indentfirst}
\usepackage[colorlinks=true,linkcolor=blue,citecolor=blue]{hyperref}

\usepackage{tikz}
\usepackage{tikz-3dplot}
\usetikzlibrary{fadings}
\tikzfading[name=fade out,
            inner color=transparent!0,
            outer color=transparent!100]
\usetikzlibrary{calc,arrows,decorations.pathreplacing,fadings,3d,positioning}

\title{Диофантовы экспоненты решёток и рост многомерных аналогов неполных частных
       \thanks{Исследование выполнено за счёт гранта Российского научного фонда № 22-21-00079, https://rscf.ru/project/22-21-00079/}}
\author{Э.\,Р.\,Бигушев, О.\,Н.\,Герман}
\date{}

\theoremstyle{definition}
\newtheorem{definition}{Определение}

\newtheorem*{notation*}{Обозначение}

\theoremstyle{remark}

\newtheorem*{remark*}{Замечание}

\theoremstyle{plain}
\newtheorem{theorem}{Теорема}

\newtheorem{proposition}{Предложение}

\newtheorem*{statement*}{Утверждение}
\newtheorem*{corollary*}{Следствие}

\newtheorem{proof_m*}{Доказательство теоремы 1}

\DeclareMathOperator{\conv}{conv}

\DeclareMathOperator{\ang}{ang}

\renewcommand{\vec}[1]{\mathbf{#1}}
\renewcommand{\geq}{\geqslant}
\renewcommand{\leq}{\leqslant}

\newcommand{\e}{\varepsilon}

\newcommand{\R}{\mathbb{R}}
\newcommand{\Z}{\mathbb{Z}}

\newcommand{\N}{\mathbb{N}}

\newcommand{\La}{\Lambda}
\newcommand{\Ga}{\Gamma}

\newcommand{\cB}{\mathcal{B}}
\newcommand{\cC}{\mathcal{C}}

\newcommand{\cI}{\mathcal{I}}

\newcommand{\cK}{\mathcal{K}}

\newcommand{\cM}{\mathcal{M}}
\newcommand{\cO}{\mathcal{O}}
\newcommand{\cP}{\mathcal{P}}

\newcommand{\cS}{\mathcal{S}}
\newcommand{\cT}{\mathcal{T}}
\newcommand{\cV}{\mathcal{V}}
\newcommand{\cW}{\mathcal{W}}

\newcommand{\starv}{\textup{St}_{\vec v}}
\newcommand{\starw}{\textup{St}_{\vec w}}
\newcommand{\staru}{\textup{St}_{\vec u}}

\begin{document}

\maketitle

\begin{abstract}
  В данной работе мы изучаем трёхмерный аналог связи экспоненты иррациональности вещественного числа с ростом его неполных частных при разложении в обыкновенную цепную дробь. В качестве многомерного обобщения цепных дробей мы рассматриваем полиэдры Клейна.
\end{abstract}

\section{Введение}\label{sec:intro}

\paragraph{\bf Экспонента иррациональности.}

Пусть $\theta$ --- вещественное число. \emph{Экспонентой иррациональности} этого числа называется величина
\begin{equation}\label{eq:mu_def}
  \mu(\theta)=\sup\Big\{\gamma\in\R\ \Big|\,\big|\theta-p/q\big|\leq|q|^{-\gamma}\text{ имеет $\infty$ решений в }(q,p)\in\Z^2 \Big\}.
\end{equation}
Для рациональных чисел экспонента иррациональности, очевидно, равна $0$, для иррациональных она не меньше $1$ ввиду теоремы Дирихле о приближении вещественного числа рациональными. Если $\theta$ иррационально, оно раскладывается в бесконечную (обыкновенную) цепную дробь $\theta=[a_0;a_1,a_2,\ldots]$ и $\mu(\theta)$ можно следующим образом связать с ростом неполных частных:
\begin{equation} \label{eq:mu_vs_partial_quotients_equality}
  \mu(\theta)=2+\limsup_{n\to\infty}\frac{\log a_{n+1}}{\log q_n}\,.
\end{equation}
Здесь $q_0,q_1,q_2,\ldots$ --- последовательность знаменателей подходящих дробей $\theta$.

Многие утверждения о цепных дробях имеют многомерные аналоги. В данной работе мы исследуем с этой точки зрения соотношение \eqref{eq:mu_vs_partial_quotients_equality}, интерпретируем его геометрически и доказываем
% его трёхмерный аналог.
трёхмерный аналог неравенства
\begin{equation} \label{eq:mu_vs_partial_quotients_inequality}
  \mu(\theta)\leq 2+\limsup_{n\to\infty}\frac{\log a_{n+1}}{\log q_n}\,.
\end{equation}

\paragraph{\bf Полигоны Клейна.}

%Одним из наиболее изящных многомерных обобщений цепных дробей являются полиэдры Клейна. Пусть $\La$ ---

Цепные дроби имеют весьма изящную геометрическую интерпретацию. Пусть $\theta_1$, $\theta_2$ --- различные вещественные числа. Рассмотрим линейные формы $L_1$, $L_2$ от двух переменных с коэффициентами, записанными в строчках матрицы
\begin{equation} \label{eq:A_theta}
  A=
  \begin{pmatrix}
    \theta_1 & -1 \\
    \theta_2 & -1
  \end{pmatrix}.
\end{equation}
Рассмотрим выпуклые оболочки
\[\cK_1=\conv\Big(\Big\{\vec z\in\Z^2\backslash\{\vec 0\}\ \Big|\,L_1(\vec z)\geq0,\ L_2(\vec z)\leq0 \Big\}\Big),\]
\[\cK_2=\conv\Big(\Big\{\vec z\in\Z^2\backslash\{\vec 0\}\ \Big|\,L_1(\vec z)\leq0,\ L_2(\vec z)\leq0 \Big\}\Big).\]
Эти выпуклые оболочки (см. рис.\ref{fig:KP_and_CF}), равно как и $-\cK_1$, $-\cK_2$, называются \emph{полигонами Клейна}.

\begin{figure}[h]
  \centering
  \begin{tikzpicture}[scale=1.4]
    \draw[very thin,color=gray,scale=1] (-3.8,-1.7) grid (4.8,4.56);

    \draw[color=black] plot[domain=-13/9:3.6] (\x, {11*\x/8}) node[right]{$y=\theta_1x$};
    \draw[color=black] plot[domain=-4:5] (\x, {-3*\x/8}) node[right]{$y=\theta_2x$};

    \fill[blue!10!,path fading=north]
        (4.8,4.56) -- (3+0.4,4+0.4*7/5) -- (3,4) -- (1,1) -- (1,0) -- (4.8,4.56) -- cycle;
    \fill[blue!10!,path fading=east]
        (4.8,4.56) -- (1,0) -- (3,-1) -- (3+1.8,-1-1.8*2/5) -- cycle;
    \fill[blue!10!,path fading=north]
        (2+1.17,3+1.17*4/3) -- (2,3) -- (0,1) -- (-2-1.8,3+1.17*4/3) -- cycle;
    \fill[blue!10!,path fading=west]
        (-2-1.8,3+1.17*4/3) -- (0,1) -- (-2,1) -- (-2-1.8,1+1.8/3) -- cycle;
    \fill[blue!10!,path fading=south]
        (-0.7,-1.7) -- (0,-1) -- (2,-1) -- (4.1,-1.7) -- cycle;
    \fill[blue!10!,path fading=west]
        (-3.8,-1.7) -- (-1,0) -- (-3,1) -- (-3.8,1+0.8*2/5) -- cycle;
    \fill[blue!10!,path fading=south]
        (-1-0.7*2/3,-1.7) -- (-1,-1) -- (-1,0) -- (-3.8,-1.7) -- cycle;

    \draw[color=blue] (3+0.4,4+0.4*7/5) -- (3,4) -- (1,1) -- (1,0) -- (3,-1) -- (3+1.8,-1-1.8*2/5);
    \draw[color=blue] (2+1.17,3+1.17*4/3) -- (2,3) -- (0,1) -- (-2,1) -- (-2-1.8,1+1.8/3);
    \draw[color=blue] (-0.7,-1.7) -- (0,-1) -- (2,-1) -- (4.1,-1.7);
    \draw[color=blue] (-1-0.7*2/3,-1.7) -- (-1,-1) -- (-1,0) -- (-3,1) -- (-3.8,1+0.8*2/5);

    \node[fill=blue,circle,inner sep=1.2pt] at (3,4) {};
    \node[fill=blue,circle,inner sep=1.2pt] at (1,1) {};
    \node[fill=blue,circle,inner sep=1.2pt] at (1,0) {};
    \node[fill=blue,circle,inner sep=1.2pt] at (3,-1) {};
    \node[fill=blue,circle,inner sep=1.2pt] at (2,3) {};
    \node[fill=blue,circle,inner sep=1.2pt] at (0,1) {};
    \node[fill=blue,circle,inner sep=1.2pt] at (-2,1) {};

    \node[fill=blue,circle,inner sep=1.2pt] at (-1,-1) {};
    \node[fill=blue,circle,inner sep=1.2pt] at (-1,0) {};
    \node[fill=blue,circle,inner sep=1.2pt] at (-3,1) {};
    \node[fill=blue,circle,inner sep=1.2pt] at (0,-1) {};
    \node[fill=blue,circle,inner sep=1.2pt] at (2,-1) {};

    \node[fill=blue,circle,inner sep=0.8pt] at (1,-1) {};
    \node[fill=blue,circle,inner sep=0.8pt] at (-1,1) {};
    \node[fill=blue,circle,inner sep=0.8pt] at (1,2) {};

    \node[right] at (1-0.03,0.5) {$a_0$};
    \node[right] at (2-2/13,2.5-3/13) {$a_2$};

    \node[right] at (2-0.06,-0.4) {$a_{-2}$};

    \node[above left] at (1.08,2-0.02) {$a_1$};

    \node[above] at (-1,0.95) {$a_{-1}$};

    \draw[blue] ([shift=({atan(-1/2)}:0.2)]1,0) arc (atan(-1/2):90:0.2);
    \draw[blue] ([shift=({atan(-2/5)}:0.2)]3,-1) arc (atan(-2/5):90+atan(2):0.2);
    \draw[blue] ([shift=(-90:0.2)]1,1) arc (-90:atan(3/2):0.2);
    \draw[blue] ([shift=({-90-atan(2/3)}:0.2)]3,4) arc (-90-atan(2/3):atan(7/5):0.2);
    \draw[blue] ([shift=({atan(4/3)}:0.2)]2,3) arc (atan(4/3):225:0.2);
    \draw[blue] ([shift=(45:0.2)]0,1) arc (45:180:0.2);
    \draw[blue] ([shift=(0:0.2)]-2,1) arc (0:90+atan(3):0.2);

    \node[right] at (3.15,-0.88) {$a_{-3}$};
    \node[right] at (1.17,0) {$a_{-1}$};
    \node[right] at (1.17,1) {$a_1$};
    \node[right] at (3.08,3.75) {$a_3$};
    \node[left] at (2-0.1,3.23) {$a_2$};
    \node[above] at (-0.05,1.18) {$a_0$};
    \node[above] at (-2.02,1.16) {$a_{-2}$};

    \draw (3,1.22) node[right]{$\cK_1$};
    \draw (-1,3.22) node[right]{$\cK_2$};
  \end{tikzpicture}
  \caption{Полигоны Клейна для}
          {$\theta_1=[a_0;a_1,a_2,\ldots],\ -1/\theta_2=[a_{-1};a_{-2},a_{-3},\ldots]$}
  \label{fig:KP_and_CF}
\end{figure}

Целочисленно--комбинаторная структура границ $\partial\cK_1$ и $\partial\cK_2$ тесно связана с цепными дробями чисел $\theta_1$ и $\theta_2$. Подробное описание этой связи можно найти, к примеру, в работах \cite{german_tlyustangelov_mjcnt}, \cite{korkina_2dim} (см. также книгу \cite{karpenkov_book}). Здесь же мы лишь укажем на следующий ключевой факт. Допустим, справедливо
\begin{equation}\label{eq:reduced_thetas}
  \theta_1>1,\quad
  -1<\theta_2<0.
\end{equation}
Тогда координаты вершин $\cK_1$ и $\cK_2$ равны (с точностью до знака) знаменателям и числителям подходящих дробей чисел $\theta_1$ и $\theta_2$, а целочисленные длины рёбер равны соответствующим неполным частным этих чисел. Напомним, что \emph{целочисленной длиной} целочисленного отрезка (т.е. отрезка, концы которого имеют целые координаты) называется количество пустых целочисленных отрезков, в нём содержащихся. Таким образом, каждому ребру полигона Клейна <<приписывается>> соответствующее ему неполное частное. Более того, неполные частные можно приписать и всем вершинам. Причина этого в том, что существует биекция (см. рис.\ref{fig:edge_vs_sprout}) между вершинами $\cK_1$ и рёбрами $\cK_2$, при которой вершина $\vec v$ соответствует ребру целочисленной длины
\begin{equation}\label{eq:alpha}
  \ang(\vec v)=|\det(\vec r_1,\vec r_2)|,
\end{equation}
где $\vec r_1$ и $\vec r_2$ суть примитивные целочисленные векторы, параллельные рёбрам, инцидентным вершине $\vec v$. На рис.\ref{fig:edge_vs_sprout} имеем $\vec r_1=\vec w-\vec v$, $\vec r_2=\vec u-\vec v$. Величину $\ang(\vec v)$ называют \emph{целочисленным углом} при вершине $\vec v$.

\begin{figure}[h]
  \centering
  \begin{tikzpicture}[scale=1.4]
    \draw[color=black] plot[domain=-3:3] (\x, {8*\x/11}) node[above right]{$y=\theta_1x$};
    \draw[color=black] plot[domain=-5:5] (\x, {-3*\x/7}) node[below]{$y=\theta_2x$};

    \fill[blue!10!,path fading=east]
        (4.6,1.8) -- (1,0) -- (4.6,-1.8) -- cycle;
    \fill[blue!10!,path fading=north]
        (2.8,1.8) -- (1,0) -- (4.6,1.8) -- cycle;
    \fill[blue!10!,path fading=west]
        (-4.6,-1.8) -- (-1,0) -- (-4.6,1.8) -- cycle;
    \fill[blue!10!,path fading=south]
        (-2.8,-1.8) -- (-1,0) -- (-4.6,-1.8) -- cycle;
    \fill[blue!10!,path fading=north]
        (-4.5,2) -- (-2,1) -- (1,1) -- (2.5,2) -- cycle;
    \fill[blue!10!,path fading=south]
        (4.5,-2) -- (2,-1) -- (-1,-1) -- (-2.5,-2) -- cycle;

    \draw[color=blue] (2.8,1.8) -- (1,0) -- (4.6,-1.8);
    \draw[color=blue] (-2.8,-1.8) -- (-1,0) -- (-4.6,1.8);
    \draw[color=blue] (-4.5,2) -- (-2,1) -- (1,1) -- (2.5,2);
    \draw[color=blue] (4.5,-2) -- (2,-1) -- (-1,-1) -- (-2.5,-2);

    \draw[very thick,color=blue] (-2,1) -- (1,1);
    \draw[very thick,color=blue] (1,0) -- (4,0);

    \draw[dashed,color=blue] (-3,1) -- (-2,1) -- (0,0) -- (1,1) -- (2,1) -- (4,0) -- (3,-1) -- (2,-1);
    \draw[dashed,color=blue] (0,0) -- (1,0);

    \node[fill=blue,circle,inner sep=1.2pt] at (-4,0) {};
    \node[fill=blue,circle,inner sep=1.2pt] at (-3,0) {};
    \node[fill=blue,circle,inner sep=1.2pt] at (-2,0) {};
    \node[fill=blue,circle,inner sep=1.2pt] at (-1,0) {};
    \node[fill=blue,circle,inner sep=1.2pt] at (0,0) {};
    \node[fill=blue,circle,inner sep=1.2pt] at (1,0) {};
    \node[fill=blue,circle,inner sep=1.2pt] at (2,0) {};
    \node[fill=blue,circle,inner sep=1.2pt] at (3,0) {};
    \node[fill=blue,circle,inner sep=1.2pt] at (4,0) {};

    \node[fill=blue,circle,inner sep=1.2pt] at (-3,1) {};
    \node[fill=blue,circle,inner sep=1.2pt] at (-2,1) {};
    \node[fill=blue,circle,inner sep=1.2pt] at (-1,1) {};
    \node[fill=blue,circle,inner sep=1.2pt] at (0,1) {};
    \node[fill=blue,circle,inner sep=1.2pt] at (1,1) {};
    \node[fill=blue,circle,inner sep=1.2pt] at (2,1) {};

    \node[fill=blue,circle,inner sep=1.2pt] at (-2,-1) {};
    \node[fill=blue,circle,inner sep=1.2pt] at (-1,-1) {};
    \node[fill=blue,circle,inner sep=1.2pt] at (0,-1) {};
    \node[fill=blue,circle,inner sep=1.2pt] at (1,-1) {};
    \node[fill=blue,circle,inner sep=1.2pt] at (2,-1) {};
    \node[fill=blue,circle,inner sep=1.2pt] at (3,-1) {};

    \draw (-0.02,0) node[above]{$\vec 0$};
    \draw (1.05,-0.05) node[above left]{$\vec v$};
    \draw (2.08,1.03) node[right]{$\vec u$};
    \draw (3.08,-1+0.03) node[right]{$\vec w$};
    \draw (4-0.05,-0.05) node[above right]{$\vec u+\vec w-\vec v$};
    \draw (-3+0.1,1.02) node[below left]{$-\vec w$};
    \draw (-2+0.2,1) node[above]{$\vec v-\vec w$};
    \draw (1-0.2,1) node[above]{$\vec u-\vec v$};
    \draw (2-0.2,-1) node[below]{$\vec w-\vec v$};

    \draw (4.1,1.2) node[above]{$\cK_1$};
    \draw (2,2) node[left]{$\cK_2$};

%    \draw (-0.5,1) node[below]{edge};
%    \draw (2.5,0) node[below]{sprout};
  \end{tikzpicture}
  \caption{Соответствие между рёбрами и вершинами} \label{fig:edge_vs_sprout}
\end{figure}

Таким образом, полигоны Клейна, оснащённые целочисленными длинами рёбер и целочисленными углами при вершинах, можно рассматривать как геометрическую интерпретацию цепных дробей. Отметим, что если условие \eqref{eq:reduced_thetas} не выполняется, указанное выше соответствие будет нарушаться лишь в некоторой окрестности начала координат. Стало быть, если обозначить через $\cV=\cV(\pm\cK_1,\pm\cK_2)$ множество вершин полигонов Клейна $\cK_1$, $-\cK_1$, $\cK_2$, $-\cK_2$, то \eqref{eq:mu_vs_partial_quotients_equality} можно переписать следующим образом:
\begin{equation} \label{eq:mu_vs_alpha}
  \max\big(\mu(\theta_1),\mu(\theta_2)\big)=
%  \limsup_{n\to\infty}\frac{\log a_{n+1}}{\log q_n}=
  2+\displaystyle\limsup_{\substack{ \vec v\in \cV,\ |\vec v|>1 \\ |\vec v|\to\infty }}\frac{\log(\ang(\vec v))}{\log|\vec v|}\,.
\end{equation}

\paragraph{Диофантова экспонента решётки.}

Для решёток в произвольной размерности определено понятие диофантовой экспоненты. Положим для каждого $\vec x=(x_1,\ldots,x_n)\in\R^n$
\begin{equation*} %\label{eq:Pi_ndim}
  \Pi(\vec x)=|x_1\cdot\ldots\cdot x_n|^{1/n}.
\end{equation*}

\begin{definition}
  Пусть $\La$ --- произвольная решётка в $\R^n$ полного ранга.
  \emph{Диофантовой экспонентой} этой решётки называется величина
  \begin{equation}\label{eq:omega_def}
    \omega(\La)=\sup\Big\{\gamma\in\R\ \Big|\,\exists\,\infty\,\vec x\in\La:\,\Pi(\vec x)\leq|\vec x|^{-\gamma} \Big\},
  \end{equation}
  где $|\cdot|$ означает sup-норму.
\end{definition}

Вернёмся к цепным дробям чисел $\theta_1$, $\theta_2$ и полигонам Клейна $\cK_1$, $\cK_2$. Пусть, как и прежде, $\cV$ --- множество вершин $\pm\cK_1$, $\pm\cK_2$. Пусть матрица $A$ задаётся \eqref{eq:A_theta}. Рассмотрим решётку
\begin{equation*} %\label{eq:La_theta}
  \La=
  A\Z^2=
  \Big\{\big(L_1(\vec z),L_2(\vec z)\big)\,\Big|\,\vec z\in\Z^2 \Big\}.
\end{equation*}
%Для каждой точки $\vec x=\big(L_1(\vec z),L_2(\vec z)\big)\in\La$ положим
%\begin{equation*} %\label{eq:Pi_2dim}
%  \Pi(\vec x)=\big|L_1(\vec z)\cdot L_2(\vec z)\big|^{1/2}.
%\end{equation*}
%%
%%\begin{definition}
%  \emph{Диофантовой экспонентой} решётки $\La$ называется величина
%  \[
%    \omega(\La)=\sup\Big\{\gamma\in\R\ \Big|\,\exists\,\infty\,\vec x\in\La:\,\Pi(\vec x)\leq|\vec x|^{-\gamma} \Big\},
%  \]
%  где $|\cdot|$ означает sup-норму.
%%\end{definition}
%
Из соображений выпуклости точки множества $\cV$ являются локальными минимумами функции $L_1(\vec x)\cdot L_2(\vec x)$ (рассматриваемой на объединении полигонов $\pm\cK_1,\pm\cK_2$). Следовательно, в определении $\omega(\La)$ (см. \eqref{eq:omega_def}) множество всех точек решётки $\La$ можно заменить на множество $A(\cV)$:
\begin{equation} \label{eq:omega_along_V}
  \omega(\La)=\sup\Big\{\gamma\in\R\ \Big|\,\exists\,\infty\,\vec x\in A(\cV):\,\Pi(\vec x)\leq|\vec x|^{-\gamma} \Big\}.
\end{equation}
Далее, если $\vec v=(q,p)$ --- точка из $\cV$, то для $\vec w=A\vec v=\big(L_1(\vec v),L_2(\vec v)\big)$ имеем при достаточно большом $|q|$
\begin{multline*} %\label{eq:Pi_vs_q_and_p}
  \Pi(\vec w)^2=
%  |L_1(\vec v)|\cdot|L_2(\vec v)|=
  |q\theta_1-p|\cdot|q\theta_2-p|\asymp \\ \asymp
  |q|\min\big(|q\theta_1-p|,|q\theta_2-p|\big)= \\ =
  q^2\min\big(|\theta_1-p/q|,|\theta_2-p/q|\big)
\end{multline*}
и
\begin{equation*} %\label{eq:x_z_q_are_same}
  |\vec w|\asymp|\vec v|\asymp|q|.
\end{equation*}
Стало быть, при достаточно большом $|q|$
\begin{equation}\label{eq:gamma_vs_2+2gamma}
  \Pi(\vec w)\asymp|\vec w|^{-\gamma}
  \iff
  \min\big(|\theta_1-p/q|,|\theta_2-p/q|\big)\asymp q^{-2-2\gamma}
\end{equation}
Эквивалентность \eqref{eq:gamma_vs_2+2gamma} ввиду \eqref{eq:omega_along_V} и \eqref{eq:mu_def} приводит к равенству
\begin{equation} \label{eq:omega_vs_mu}
  \max\big(\mu(\theta_1),\mu(\theta_2)\big)=2+2\omega(\La).
%  \omega(\La)=\frac12\max\big(\mu(\theta_1),\mu(\theta_2)\big)-1,
\end{equation}
Комбинируя \eqref{eq:mu_vs_alpha} и \eqref{eq:omega_vs_mu}, получаем геометрическую интерпретацию равенства \eqref{eq:mu_vs_partial_quotients_equality}:

%\begin{proposition}
  \begin{equation*}%\label{eq:omega_vs_alpha}
    \omega(\La)=
    \frac12
    \displaystyle\limsup_{\substack{ \vec v\in \cV,\ |\vec v|>1 \\ |\vec v|\to\infty }}\frac{\log(\ang(\vec v))}{\log|\vec v|}\,.
  \end{equation*}
%\end{proposition}

Соответственно, целью данной работы является обобщение на трёхмерный случай неравенства

\begin{equation} \label{eq:omega_vs_alpha_leq}
  \omega(\La)\leq
  \frac12
  \displaystyle\limsup_{\substack{ \vec v\in \cV,\ |\vec v|>1 \\ |\vec v|\to\infty }}\frac{\log(\ang(\vec v))}{\log|\vec v|}\,.
\end{equation}

\section{Формулировка основного результата}

Пусть $L_1$, $L_2$, $L_3$ --- линейно независимые линейные формы от трёх переменных.
%Рассмотрим решётку
%\begin{equation} \label{eq:La_3d}
%  \La=
%  \Big\{\big(L_1(\vec z),L_2(\vec z),L_3(\vec z)\big)\,\Big|\,\vec z\in\Z^3 \Big\}.
%\end{equation}
%Как и в предыдущем параграфе, для каждой точки $\vec x=\big(L_1(\vec z),L_2(\vec z),L_3(\vec z)\big)$ решётки $\La$ положим
%\begin{equation*} %\label{eq:Pi_2dim}
%  \Pi(\vec x)=\big|L_1(\vec z)\cdot L_2(\vec z)\cdot L_3(\vec z)\big|^{1/3}.
%\end{equation*}
%%
%%\begin{definition}
%%  \emph{Диофантовой экспонентой} решётки $\La$ называется величина
%  Величина
%  \[
%    \omega(\La)=\sup\Big\{\gamma\in\R\ \Big|\,\exists\,\infty\,\vec x\in\La:\,\Pi(\vec x)\leq|\vec x|^{-\gamma} \Big\},
%  \]
%где, как и прежде, $|\cdot|$ означает sup-норму,
%также называется \emph{диофантовой экспонентой} решётки $\La$.
%%\end{definition}
%
Нулевые подпространства форм $L_1$, $L_2$, $L_3$ разбивают $\R^3$ на восемь симплициальных (замкнутых) конусов. В каждом из них рассмотрим выпуклую оболочку ненулевых целых точек. Обозначим их $\cK_i$, $i=1,\ldots,8$ (в произвольном порядке). Эти выпуклые оболочки называются \emph{полиэдрами Клейна} (соответствующими линейным формам $L_1$, $L_2$, $L_3$).

Если формы $L_1$, $L_2$, $L_3$ не обращаются в нуль в ненулевых целых точках, то, как показано в \cite{moussafir}, все $\cK_i$ суть обобщённые многогранники (т.е. их пересечения с любыми компактными многогранниками также являются многогранниками). В частности, в этом случае в каждой вершине полиэдра Клейна сходится конечное число рёбер. Стало быть, мы можем определить для рёберной звезды вершины полиэдра Клейна многомерный аналог целочисленного угла. Вообще говоря, такой многомерный аналог можно определять по-разному. Мы воспользуемся понятием \emph{определителя} рёберной звезды, введённым в работах \cite{german_2006}, \cite{german_2007}.

% Мы рассмотрим два способа определить многомерный аналог целочисленного угла при вершине $\cK$. Первый способ восходит к работам \cite{german_2006}, \cite{german_2007}. Второй основывается на идее понимать целочисленный угол как отношение \eqref{eq:alpha_def_as_a_fraction} и оказывается более подходящим для обобщения на трёхмерный случай равенства \eqref{eq:mu_vs_partial_quotients_equality}.

Пусть $\cK$ --- один из $\cK_i$ и пусть $\vec v$ --- какая-то его вершина. Пусть $\vec r_1,\ldots,\vec r_k$ --- примитивные целочисленные векторы, параллельные рёбрам $\cK$, инцидентным вершине $\vec v$. Обозначим через $\starv$ рёберную звезду вершины $\vec v$.

\begin{definition}\label{def:starv}
  \emph{Определителем} рёберной звезды $\starv$ называется величина
  \[
    \det\starv=\sum_{1\leq i_1<i_2<i_3\leq k}|\det(\vec r_{i_1},\vec r_{i_2},\vec r_{i_3})|.
  \]
\end{definition}

Нетрудно доказать, что определяемый таким образом $\det\starv$ равняется объёму суммы Минковского отрезков $[\vec 0,\vec r_1],\ldots,[\vec 0,\vec r_k]$. Ясно также, что аналогичная величина в двумерном случае в точности совпадает с целочисленным углом при вершине (см. \eqref{eq:alpha}).

Следующее утверждение являются основным результатом данной работы.

\begin{theorem}\label{t:omega_vs_det_starv}
  Рассмотрим решётку
  \begin{equation} \label{eq:La_3d}
    \La=
    \Big\{\big(L_1(\vec z),L_2(\vec z),L_3(\vec z)\big)\,\Big|\,\vec z\in\Z^3 \Big\}.
  \end{equation}
  Пусть % решётка $\La$ определяется соотношением \eqref{eq:La_3d} и пусть
  формы $L_1$, $L_2$, $L_3$ не обращаются в нуль в ненулевых целых точках. Пусть $\cK_1,\ldots,\cK_8$ --- полиэдры Клейна, соответствующие этим линейным формам, и пусть $\cV$ --- множество всех их вершин. Тогда
  \begin{equation} \label{eq:omega_vs_det_starv}
    \omega(\La)
    \leq
    \frac23
    \displaystyle\limsup_{\substack{ \vec v\in \cV,\ |\vec v|>1 \\ |\vec v|\to\infty }}\frac{\log(\det\starv)}{\log|\vec v|}\,.
  \end{equation}
%  \begin{equation} \label{eq:omega_vs_det_starv}
%    \frac32
%    \omega(\La)
%    \leq
%    \displaystyle\limsup_{\substack{ \vec v\in \cV,\ |\vec v|>1 \\ |\vec v|\to\infty }}\frac{\log(\det\starv)}{\log|\vec v|}
%    \leq
%    \frac??
%    \omega(\La).
%  \end{equation}
\end{theorem}

% Отметим, что теорема \ref{t:omega_vs_det_starv} мгновенно следует из теоремы \ref{t:omega_vs_inndet_starv}, ибо внутренний определитель рёберной звезды никогда не превосходит её определителя. Заменить знак неравенства на знак равенства в \eqref{eq:omega_vs_det_starv} судя по всему нельзя. В параграфе \ref{sec:counterexample} мы описываем семейство соответствующих контрпримеров. Сейчас же мы докажем теорему \ref{t:omega_vs_inndet_starv}.

\section{Доказательство теоремы \ref{t:omega_vs_det_starv}}\label{sec:proof}

По аналогии с двумерным случаем обозначим через $A$ матрицу $3\times3$, в строках которой записаны коэффициенты линейных форм $L_1$, $L_2$, $L_3$. Не ограничивая общности, можно считать, что $\det A=1$. Тогда
\[
  \La=A\Z^3,\qquad\det\La=1.
\]
Положим
\[
  \cK_i'=A\cK_i.
\]
Тогда $\cK_1',\ldots,\cK_8'$ суть выпуклые оболочки ненулевых точек решётки $\La$ в каждом из $8$ ортантов. Они называются \emph{полиэдрами Клейна} решётки $\La$. Рассмотрим $\cW=A(\cV)$ --- множество всех их вершин. Как и в двумерном случае, воспользуемся тем, что точки множества $\cW$ являются локальными минимумами функции $\Pi(\vec x)$ (рассматриваемой на объединении полиэдров $\cK_1',\ldots,\cK_8'$). Тогда
\begin{equation} \label{eq:omega_as_limsup}
  \omega(\La)=
  \limsup_{\substack{\vec x\in\La,\ |\vec x|>1 \\ |\vec x|\to\infty}}\frac{\log\big(\Pi(\vec x)^{-1}\big)}{\log|\vec x|}=
  \limsup_{\substack{\vec w\in\cW,\ |\vec w|>1 \\ |\vec w|\to\infty}}\frac{\log\big(\Pi(\vec w)^{-1}\big)}{\log|\vec w|}\,.
\end{equation}
Каждой точке $\vec w\in\cW$ соответствует точка $\vec v\in\cV$, такая что $\vec w=A\vec v$. При таком соответствии $\starw=A(\starv)$,
% и $\inndet\starw=(\det A)^{-1/2}\cdot\inndet\starv$,
то есть
\[
  \det\starw=\det\starv
  \qquad\text{ и }\qquad
  |\vec w|\asymp|\vec v|
  \text{ при }|\vec v|\to\infty.
\]
Таким образом, ввиду \eqref{eq:omega_as_limsup}, для доказательства \eqref{eq:omega_vs_det_starv} достаточно показать, что
\begin{equation} \label{eq:Pi_leq_starw_plus_o}
  \frac{\log\big(\Pi(\vec w)^{-1}\big)}{\log|\vec w|}
  \leq
  \frac23\cdot
  \frac{\log(\det\starw)}{\log|\vec w|}+o(1)
\end{equation}
при $\vec w\in\cW$, $|\vec w|\to\infty$. На самом деле мы докажем нечто большее. Справедливо следующее \emph{локальное} утверждение.

\begin{proposition} \label{prop:starw_geq_cPi}
  Существует такая положительная константа $c$,
%  Существуют такие положительные константы $c_1$, $c_2$,
%  зависящие только от определителя решётки $\La$,
  что для любой точки $\vec w\in\cW$ справедливо
  \begin{equation} \label{eq:starw_geq_cPi}
    \det\starw\geq c\Pi(\vec w)^{-3/2}.
%    c_1\Pi(\vec w)^{-3/2}\leq\inndet\starw\leq c_2\Pi(\vec w)^{-3/2}.
  \end{equation}
\end{proposition}

\begin{proof}
  Положим $\e$ равным произвольному достаточно малому положительному числу, например,
  \[
    \e=2^{-100}.
  \]
  Если $\Pi(\vec w)\geq\e$, то при $c_1=\e^{3/2}$ справедливо $\det\starw\geq1\geq c_1\Pi(\vec w)^{-3/2}$. Будем далее считать, что
  \begin{equation}\label{eq:suppose_Pi_is_small}
    \Pi(\vec w)<\e.
  \end{equation}

  Разобьём наше рассуждение на несколько шагов.

  \paragraph{Шаг 1. Гиперболический поворот.}

  Пусть $\vec w=(w_1,\ldots,w_d)$. Рассмотрим диагональный оператор
  \[
    D=\textup{diag}\Big(\Pi(\vec w)\big/w_1,\ldots,\Pi(\vec w)\big/w_d\Big).
  \]
  Положим
  \[
    \vec u=D\vec w,
    \qquad
    \La_{\vec w}=D\La,
    \qquad
    \cK_{\vec w}=D\cK_i',
  \]
  где $\cK_i'$ --- тот полиэдр, вершиной которого является $\vec w$. Тогда $\cK_{\vec w}$ --- полиэдр Клейна решётки $\La_{\vec w}$, соответствующий положительному ортанту (то есть конусу, состоящему из точек с неотрицательными координатами), и
  \[
    \det\La_{\vec w}=\det\La=1,
    \qquad
    \det\staru=\det\starw,
  \]
  поскольку $|\det D|=1$. При этом все координаты точки $\vec u$ равны $\Pi(\vec w)$. В частности,
  \[
    |\vec u|=\Pi(\vec u)=\Pi(\vec w)<\e.
  \]

  \paragraph{Шаг 2. Короткие и длинные векторы.}

  Покажем, что вектор $\vec u$ является кратчайшим вектором решётки $\La_{\vec w}$ в sup-норме. Рассмотрим куб $\cB$, являющийся замкнутым шаром радиуса $|\vec u|$ в sup-норме. Это куб с центром в начале координат $\vec 0$ с ребром $2|\vec u|$. Поскольку $\vec u$ --- вершина $\cK_{\vec w}$, существует опорная к $\cK_{\vec w}$ гиперплоскость, пересекающаяся с $\cK_{\vec w}$ по точке $\vec u$. Эта гиперплоскость делит куб $\vec u+\cB$ на две симметричные относительно $\vec u$ части. В той части, которая содержит точку $\vec 0$, нет точек решётки $\La_{\vec w}$, отличных от $\vec 0$ и $\vec u$. Стало быть, во всём кубе $\vec u+\cB$ нет точек решётки, отличных от $\vec 0$, $\vec u$, $2\vec u$. Значит, и куб $\cB$ не содержит точек $\La_{\vec w}$, отличных от $\vec 0$, $\vec u$, $-\vec u$. Таким образом, $\vec u$, действительно, является кратчайшим вектором решётки $\La_{\vec w}$ в sup-норме.

  Пусть $\vec p_1,\ldots,\vec p_k$ --- примитивные векторы решётки $\La_{\vec w}$, параллельные рёбрам $\cK_{\vec w}$, инцидентным вершине $\vec u$. Эти векторы --- образы при действии оператора $DA$ векторов $\vec r_1,\ldots,\vec r_k$ из определения \ref{def:starv} и
  \[
    \det\staru=\sum_{1\leq i_1<i_2<i_3\leq k}|\det(\vec p_{i_1},\vec p_{i_2},\vec p_{i_3})|.
  \]
  Упорядочим $\vec p_1,\ldots,\vec p_k$ по возрастанию их евклидовой нормы, для которой будем использовать обозначение $|\cdot|_2$. Для любого $\vec p_i$, такого что $\vec u$, $\vec p_1$, $\vec p_i$ линейно независимы, справедливо $|\vec u|_2|\vec p_i|_2^2\geq|\vec u|_2|\vec p_1|_2|\vec p_i|_2\geq|\det(\vec u,\vec p_1,\vec p_i)|\geq\det\La_{\vec w}=1$, то есть $|\vec p_i|_2\geq|\vec u|_2^{-1/2}$, откуда
  \begin{equation}\label{eq:long_edges}
    |\vec p_i|\geq
    \frac{|\vec p_i|_2}{\sqrt3}\geq
    \frac{|\vec u|_2^{-1/2}}{\sqrt3}\geq
    \frac{\big(|\vec u|\sqrt3\big)^{-1/2}}{\sqrt3}=
    \frac{|\vec u|}{3^{3/4}|\vec u|^{3/2}}>
    \frac{|\vec u|}{3^{3/4}\e^{3/2}}>
    \frac{|\vec u|}{4\e}\,.
  \end{equation}

  \paragraph{Шаг 3. Полигон Клейна.}

  Рассмотрим двумерное подпространство $\pi$, порождённое векторами $\vec u$, $\vec p_1$ и решётку $\Ga=\pi\cap\La_{\vec w}$.
%  Поскольку $\vec u$ и $\vec p_1$ линейно независимы, $\Ga$ имеет ранг $2$.
  Поскольку треугольник с вершинами $\vec 0$, $\vec u$, $\vec u+\vec p_1$ не содержит точек решётки $\La_{\vec w}$, отличных от вершин, векторы $\vec u$, $\vec p_1$ образуют базис решётки $\Ga$. Обозначим через $\cC$ пересечение $\pi$ с положительным ортантом и рассмотрим $\cK_\pi=\conv(\cC\cap\Ga\backslash\{\vec 0\})$ --- полигон Клейна решётки $\Ga$, соответствующий конусу $\cC$. Тогда $\vec p_1$ --- примитивный вектор решётки $\Ga$, параллельный одному из рёбер $\cK_\pi$, инцидентных вершине $\vec u$. Обозначим через $\vec p_1'$ примитивный вектор $\Ga$, параллельный второму из этих рёбер. Отметим, что $\vec p_1'$, вообще говоря, не обязан быть одним из $\vec p_2,\ldots,\vec p_k$. Векторы $\vec u$, $\vec p_1'$ также образуют базис решётки $\Ga$. Стало быть, точки $\vec p_1$ и $\vec p_1'$ находятся на одинаковом расстоянии от прямой, порождённой $\vec u$. Следовательно,
  \begin{equation}\label{eq:sprout_in_action}
    \vec p_1+\vec p_1'=t\vec u,\qquad t\in\N
  \end{equation}
  (ср. с рис. \ref{fig:edge_vs_sprout}). Таким образом, если положить для $\vec x=(x_1,x_2,x_3)$
  \[
    M(\vec x)=x_1+x_2+x_3,
  \]
  то ввиду \eqref{eq:sprout_in_action} получим $M(\vec p_1+\vec p_1')\geq3|\vec u|$, откуда либо $M(\vec u+\vec p_1)\geq\frac92|\vec u|$, либо $M(\vec u+\vec p_1')\geq\frac92|\vec u|$.

  Учитывая \eqref{eq:long_edges}, заключаем, что среди точек $\vec u+\vec p_i$, $i=1,\ldots,k$, неравенству
  \begin{equation}\label{eq:almost_empty_zone}
    M(\vec x)<\frac92|\vec u|
  \end{equation}
  удовлетворяет не более чем одна, причём это или $\vec u+\vec p_1$, или $\vec u+\vec p_2$ (последнее возможно лишь тогда, когда $\vec p_1'=\vec p_2$).

  \paragraph{Шаг 4. Сечение и выпуклая оболочка.}

  Положим
  \[
    M_1(\vec x)=x_1+x_2+\frac94x_3,\qquad
    M_2(\vec x)=x_1+\frac94x_2+x_3,\qquad
    M_3(\vec x)=\frac94x_1+x_2+x_3.
  \]
  Для любой точки $\vec x=(x_1,x_2,x_3)$, удовлетворяющей неравенствам
  \begin{equation}\label{eq:hexahedron}
    x_1,x_2,x_3\geq0,\qquad
    M_1(\vec x),M_2(\vec x),M_3(\vec x)\leq\frac92|\vec u|,
  \end{equation}
  справедливо $x_1,x_2,x_3\leq2|\vec u|$. Но, как мы заметили выше, в кубе $\vec u+\cB$ нет точек решётки, отличных от $\vec 0$, $\vec u$, $2\vec u$. Стало быть, никакая из точек $\vec u+\vec p_i$, $i=1,\ldots,k$, не удовлетворяет \eqref{eq:hexahedron}. Таким образом, если среди этих точек и существует такая, которая удовлетворяет \eqref{eq:almost_empty_zone}, она будет удовлетворять хотя бы одному из неравенств
  \[
    M_1(\vec x)>\frac92|\vec u|,\qquad
    M_2(\vec x)>\frac92|\vec u|,\qquad
    M_3(\vec x)>\frac92|\vec u|.
  \]
  Не ограничивая общности, можно считать, что она удовлетворяет первому из них. Тогда и для каждого $i=1,\ldots,k$ справедливо
  \[
    M_1(\vec u+\vec p_i)\geq\frac92|\vec u|,
  \]
  поскольку при неотрицательных $x_1$, $x_2$, $x_3$ справедливо $M_1(\vec x)\geq M(\vec x)$.

  Обозначим через $\cO$ положительный ортант --- конус, состоящий из точек с неотрицательными координатами. Рассмотрим плоскость
  \[
    \cM=\Big\{\vec x\in\R^3\,\Big|\,M_1(\vec x)=\frac92|\vec u| \Big\}
  \]
  и сечения
  \[
    \cS=\cM\cap\cO,\qquad
    \cT=\cM\cap(\vec u+\cO),\qquad
    \cP=\cM\cap\cK_{\vec w}.
  \]
  Тогда $\cT\subset\cP\subset\cS$, причём $\cT$ непусто, ибо
  $
    M_1(\vec u)=\frac{17}4|\vec u|<\frac92|\vec u|.
  $
  Более того, $\cT$ является треугольником с вершинами в точках
  \[
    \Big(\frac54|\vec u|,|\vec u|,|\vec u|\Big),\qquad
    \Big(|\vec u|,\frac54|\vec u|,|\vec u|\Big),\qquad
    \Big(|\vec u|,|\vec u|,\frac{10}9|\vec u|\Big).
  \]
  Множество $\cS$ является треугольником с вершинами в точках
  \[
    \Big(\frac92|\vec u|,0,0\Big),\qquad
    \Big(0,\frac92|\vec u|,0\Big),\qquad
    \Big(0,0,2|\vec u|\Big).
  \]
  Множество $\cP$ является многоугольником с вершинами в точках пересечения отрезков $[\vec u,\vec u+\vec p_i]$, $i=1,\ldots,k$, с плоскостью $\cM$.
%  При этом $\cT$ содержится в (относительной) внутренности $\cP$.
  При этом из \eqref{eq:long_edges} следует, что вершины $\cP$, соответствующие $\vec p_i$, не лежащим в подпространстве $\pi$, то есть линейно независимым с $\vec u$ и $\vec p_1$, содержатся в
  \[
    \cT_\e=\cM\cap\big((1-\e)\vec u+\cO\big).
  \]
  Вершины же, соответствующие $\vec p_i$, лежащим в $\pi$, содержатся в отрезке
  \[
    \cI=\cS\cap\pi.
  \]
  Таким образом,
  \begin{equation}\label{eq:sandwich}
    \cT\subset\cP\subset\conv(\cT_\e\cup\cI). % \subset\cS.
  \end{equation}

  \paragraph{Шаг 5. Анализ сечения.}

  Обозначим через $\vec a$ точку пересечения плоскости $\cM$ с прямой, порождённой вектором $\vec u$. Тогда координаты точки $\vec a$ равны, эта точка принадлежит отрезку $\cI$ и является центром гомотетии треугольников $\cT$, $\cT_\e$, $\cS$. Выделим в $\cS$ три зоны (см. рис. \ref{fig:the_section}):
  \begin{align*}
    \Omega_1 & =\Big\{\vec x=(x_1,x_2,x_3)\in\cS\backslash\cT\,\Big|\,x_2<\frac{x_1+x_3}2,\ x_3\leq\frac{x_1+x_2}2 \Big\}, \\ \vphantom{\bigg|}
    \Omega_2 & =\Big\{\vec x=(x_1,x_2,x_3)\in\cS\backslash\cT\,\Big|\,x_3<\frac{x_1+x_2}2,\ x_1\leq\frac{x_2+x_3}2 \Big\}, \\
    \Omega_3 & =\Big\{\vec x=(x_1,x_2,x_3)\in\cS\backslash\cT\,\Big|\,x_1<\frac{x_2+x_3}2,\ x_2\leq\frac{x_1+x_3}2 \Big\}.
  \end{align*}
  Отрезок $\cI$ имеет непустое пересечение ровно с одним из множеств $\Omega_1$, $\Omega_2$, $\Omega_3$. На рис. \ref{fig:the_section} изображён случай пересечения с $\Omega_3$. Разберём подробно этот случай. Остальные два разбираются аналогично.

  Обозначим через $\vec b$ конец отрезка $\cI$, не принадлежащий $\Omega_3$. Проведём из точки $\vec b$ прямые через нижние вершины треугольника $\cT$. Эти прямые отсекают от $\cT_\e$ треугольники $\Delta_1$ и $\Delta_2$. Каждый из них в силу \eqref{eq:sandwich} обязан содержать хотя бы одну вершину $\cP$. Пусть $\Delta_1$ содержит вершину, соответствующую $\vec p_{i_1}$, и $\Delta_2$ --- вершину, соответствующую $\vec p_{i_2}$.
%  Если интервал $\cI\cap\Omega_3$ (выделенный жирным на рис. \ref{fig:the_section}) содержит вершину $\cP$, будем считать, что она соответствует $\vec p_{i_3}$. Если же $\cI\cap\Omega_3$ не содержит вершин $\cP$,
  Покажем, что $\Omega_3$ также содержит вершину $\cP$. Если это не так, то в соотношении \eqref{eq:sandwich} отрезок $\cI$ можно заменить отрезком $[\vec a,\vec b]$:
  \[
    \cT\subset\cP\subset\conv(\cT_\e\cup[\vec a,\vec b]).
  \]
  Тогда из предположения, что $\Omega_3$ не содержит вершин $\cP$, следует, что верхняя вершина $\cT$ не принадлежит $\cP$, что противоречит включению $\cT\subset\cP$. Стало быть, в $\Omega_3$ найдётся вершина $\cP$. Пусть она соответствует $\vec p_{i_3}$.

%  Проведём через верхнюю вершину $\cT$ прямую, параллельную противоположной стороне $\cT$. Эта прямая отсекает от $\cT_\e$ треугольник $\Delta_3$. Если интервал $\cI\cap\Omega_3$ (выделенный жирным на рис. \ref{fig:the_section}) не содержит вершин $\cP$, то в соотношении \eqref{eq:sandwich} отрезок $\cI$ можно заменить отрезком $[\vec a,\vec b]$:
%  \[
%    \cT\subset\cP\subset\conv(\cT_\e\cup[\vec a,\vec b]).
%  \]
%  Стало быть, в этом случае треугольник $\Delta_3$ содержит вершину $\cP$. Пусть эта вершина соответствует $\vec p_{i_3}$. Если же $\cI\cap\Omega_3$ всё-таки содержит вершину $\cP$, возьмём эту вершину и будем считать, что она соответствует $\vec p_{i_3}$.

\begin{figure}[h]
  \centering
  \begin{tikzpicture}[scale=0.47]
    \fill[blue!10!]
      (0,8/5) -- (-4/5*8/9,4/5) -- (-4*8/9,4) -- (0,8) -- (4*8/9,4) -- (4/5*8/9,4/5) -- cycle;
    \fill[blue!10!]
      (-8/5*2,-2) -- (-2*8/9,-2) -- (-10*8/9,-10) -- (-16,-10) -- (-16+10*8/9,0) -- (-2*4/5*8/9,0) -- cycle;
    \fill[blue!10!]
      (8/5*2,-2) -- (2*8/9,-2) -- (10*8/9,-10) -- (16,-10) -- (16-10*8/9,0) -- (2*4/5*8/9,0) -- cycle;

    \draw[color=gray,very thin] (-16,-10) -- (0,0) -- (16,-10);
    \draw[color=gray,very thin] (0,8) -- (0,0);

    \draw[color=gray] (-16,-10) -- (0,8) -- (16,-10) -- cycle;

    \draw[color=blue] (-8/5*2,-2) -- (0,8/5) -- (8/5*2,-2) -- cycle;

    \draw[color=blue,very thin] (-4,-2.5) -- (0,2) -- (4,-2.5) -- cycle;

    \draw[color=blue,very thin,dashed] (0,0) -- (4/5*8/9,4/5);
    \draw[color=blue,very thin] (4*8/9,4) -- (4/5*8/9,4/5);
    \draw[color=blue,very thin,dashed] (0,0) -- (-10*8/9,-10);

    \draw[color=blue,very thin,dashed] (0,0) -- (-4*8/9,4);
    \draw[color=blue,very thin,dashed] (0,0) -- (2*8/9,-2);
    \draw[color=blue,very thin] (10*8/9,-10) -- (2*8/9,-2);

    \draw[color=blue,very thin] (-64/9,0) -- (-8/5*8/9,0);
    \draw[color=blue,very thin,dashed] (-8/5*8/9,0) -- (64/9,0);

%    \draw[color=blue,very thin] (0,0) -- (64/25,128/25);
%    \draw[color=blue,very thin] (0,0) -- (-5,-10);
%    \draw[color=blue,thick] (64/125,128/125) -- (64/25,128/25);
    \draw[color=blue] (-5,-10) -- (64/25,128/25);

%    \fill[blue!30!]
%      (0,5/4*8/5) -- (1/4*8/5*8/9,8/5) -- (-1/4*8/5*8/9,8/5) -- cycle;
%    \draw[color=blue,thick] (0,2) -- (1/4*8/5*8/9,8/5) -- (-1/4*8/5*8/9,8/5) -- cycle;

    \draw[color=gray,very thin] (-5,-10) -- (8/5*2,-2);
    \fill[blue!30!]
      (4,-2.5) -- (215/80,-2.5) -- (2336/689,-1250/689) -- cycle;
    \draw[color=blue,thick] (4,-2.5) -- (215/80,-2.5) -- (2336/689,-1250/689) -- cycle;

    \draw[color=gray,very thin] (-5,-10) -- (-8/5*2,-2);
    \fill[blue!30!]
      (-4,-2.5) -- (-53/16,-2.5) -- (-736/239,-350/239) -- cycle;
    \draw[color=blue,thick] (-4,-2.5) -- (-53/16,-2.5) -- (-736/239,-350/239) -- cycle;

    \draw (-10,-8) node[above]{$\Omega_1$};
    \draw (10,-8) node[above]{$\Omega_2$};
    \draw (0,3.5) node[above]{$\Omega_3$};

    \node[fill=blue,circle,inner sep=1pt] at (0,0) {};
    \node[fill=blue,circle,inner sep=1pt] at (-5,-10) {};
    \draw (0.06,0) node[below]{$\vec a$};
    \draw (-5,-10) node[below]{$\vec b$};

    \draw[->,>=stealth',color=black,very thin] (12,-2.3) node[right]{$\Delta_2$} -- (3.43,-2.3);
    \draw[->,>=stealth',color=black,very thin] (-12,-2.3) node[left]{$\Delta_1$} -- (-3.45,-2.3);
%    \draw[->,>=stealth',color=black,very thin] (-8.42,1.73) node[left]{$\Delta_3$} -- (0.12,1.73);

    \draw[->,>=stealth',color=black,very thin] (10.75,-0.9) node[right]{$\cT$} -- (1.85,-0.9);
    \draw[->,>=stealth',color=black,very thin] (9.5,0.5) node[right]{$\cT_\e$} -- (1.15,0.5);

    \draw (-2.5,-5) node[right]{$\cI$};

%       8/(16/5+5)=40/41
%       -210/41+40/41 x = -5/2
%       40/41 x = 420/82-205/82
%       40/41 x = 215/82
%       x = 215/80
%
%       -210/41 + 40/41 x = 2 - 9/8 x
%       -210 + 40 x = 82 - 41*9/8 x
%       -1680 + 320 x = 656 - 369 x
%       689 x = 2336
%       x = 2336/689
%       y = 2 - 292*9/689 = -1250/689
%
%       8/(5-16/5)=40/9
%       110/9 + 40/9 x = -5/2
%       220 + 80 x = -45
%       80 x = -265
%       x = -265/80 = -53/16
%
%       110/9 + 40/9 x = 2 + 9/8 x
%       880 + 320 x = 144 + 81 x
%       239 x = -736
%       x = -736/239
%       y = 2 - 92*9/239 = -350/239

  \end{tikzpicture}
  \caption{Сечение плоскостью $\cM$} \label{fig:the_section}
\end{figure}

  \paragraph{Шаг 6. Непосредственная оценка.}

  Поскольку $\e$ мало, найденные векторы $\vec p_{i_1}$, $\vec p_{i_2}$, $\vec p_{i_3}$ удовлетворяют следующим двум условиям:

  1) любые два из них линейно независимы с вектором $\vec u$;

  2) $|\det(\vec p_{i_1},\vec p_{i_2},\vec p_{i_3})|\asymp|\vec p_{i_1}|\cdot|\vec p_{i_2}|\cdot|\vec p_{i_3}|$.
  \\
  Причём константы, подразумеваемые знаком ,,$\asymp$``, абсолютны. Благодаря этим свойствам получаем:
  \begin{multline*}
    \Pi(\vec u)^3\det(\vec p_{i_1},\vec p_{i_2},\vec p_{i_3})^2\asymp
    |\vec u|^3|\vec p_{i_1}|^2|\vec p_{i_2}|^2|\vec p_{i_2}|^2
    \geq \\ \geq
    |\det(\vec u,\vec p_{i_1},\vec p_{i_2})\cdot
    \det(\vec u,\vec p_{i_2},\vec p_{i_3})\cdot
    \det(\vec u,\vec p_{i_3},\vec p_{i_1})|\geq1.
  \end{multline*}
  Таким образом, существует константа $c_2>0$, такая что
  \[
    \det\starw=\det\staru\geq|\det(\vec p_{i_1},\vec p_{i_2},\vec p_{i_3})|\geq c_2\Pi(\vec u)^{-3/2}=c_2\Pi(\vec w)^{-3/2}.
  \]
  Учитывая, что это неравенство получено в предположении \eqref{eq:suppose_Pi_is_small}, остаётся положить $c=\min(c_1,c_2)$.
\end{proof}

Из предложения \ref{prop:starw_geq_cPi} следует \eqref{eq:Pi_leq_starw_plus_o} и, стало быть, и \eqref{eq:omega_vs_det_starv}. Теорема \ref{t:omega_vs_det_starv} доказана.

\section{Об обращении предложения \ref{prop:starw_geq_cPi}}\label{sec:counterexample}

Опишем пример, показывающий, что буквальное обращение предложения \ref{prop:starw_geq_cPi} невозможно. Возьмём произвольное целое $n\geq3$ и положим
\[
  \vec e_1=(n,1,0),\qquad
  \vec e_2=(0,n,1),\qquad
  \vec e_3=(1,0,n),\qquad
  \vec v=(1,1,1).
\]
Рассмотрим конус $\cC$ с вершиной в начале координат и рёбрами, порождёнными векторами $\vec e_1$, $\vec e_2$, $\vec e_3$. Полиэдр Клейна $\cK=\conv(\cC\cap\Z^3\backslash\{\vec 0\})$ имеет три ограниченные грани и три неограниченные. Ограниченные грани суть треугольники $\vec v\vec e_1\vec e_2$, $\vec v\vec e_2\vec e_3$, $\vec v\vec e_3\vec e_1$. Действительно, тетраэдр $\vec 0\vec v\vec e_1\vec e_2$ пуст (в том смысле, что он не содержит целых точек, отличных от вершин), ибо он <<зажат>> между плоскостями $x_3=0$ и $x_3=1$, в то время как векторы $\vec e_1-\vec 0$ и $\vec e_2-\vec v$ примитивны. Аналогично, пусты тетраэдры $\vec 0\vec v\vec e_2\vec e_3$ и $\vec 0\vec v\vec e_3\vec e_1$. Стало быть, в вершине $\vec v$ сходятся три грани --- $\vec v\vec e_1\vec e_2$, $\vec v\vec e_2\vec e_3$ и $\vec v\vec e_3\vec e_1$. Определитель рёберной звезды $\starv$ равен
\[
  \det\starv=
  |\det(\vec e_1-\vec v,\vec e_2-\vec v,\vec e_3-\vec v)|=
%  \det
  \left|
  \begin{matrix}
    n-1 & 1 & 0 \\
    0 & n-1 & 1 \\
    1 & 0 & n-1
  \end{matrix}
  \right|=
  (n-1)^3-1.
\]
Определим линейные формы $L_1$, $L_2$, $L_3$ строчками матрицы
\[
  A=
  (n^3+1)^{-2/3}
  \left(
  \begin{matrix}
    n^2 & 1 & -n \\
    -n & n^2 & 1 \\
    1 & -n & n^2
  \end{matrix}
  \right).
\]
Тогда $L_i(\vec e_j)=0$ при $i\neq j$ и $\det A=1$. Положим
\[
  \La=A\Z^3,\qquad\cK'=A\cK,\qquad\vec w=A\vec v.
\]
Тогда $\vec w$ --- вершина $\cK'$,
\begin{equation}\label{eq:Pi_is_large}
  \Pi(\vec w)=L_1(\vec v)L_2(\vec v)L_3(\vec v)=\frac{(n^2-n+1)^3}{(n^3+1)^2}\asymp1\quad\text{при}\quad n\to\infty
\end{equation}
и
\begin{equation}\label{eq:det_is_huge}
  \det\starw=
  \det\starv=
  (n-1)^3-1\asymp n^3\quad\text{при}\quad n\to\infty.
\end{equation}
Решётку $\La$, разумеется, можно немного <<подправить>>, чтобы она не имела ненулевых точек на координатных плоскостях. Например, для малого положительного $\e$ можно взять вместо $A$ матрицу
\[
  A_\e=
  A+
  \left(
  \begin{matrix}
    0 & 0 & \e \\
    \e & 0 & 0 \\
    0 & \e & 0
  \end{matrix}
  \right).
\]
%\[
%  A_\e=
%  (n^3+1)^{-2/3}
%  \left(
%  \begin{matrix}
%    n^2 & 1 & -n+\e \\
%    -n+\e & n^2 & 1 \\
%    1 & -n+\e & n^2
%  \end{matrix}
%  \right).
%\]
Тогда при $\e$ достаточно малом точка $\vec v$ будет оставаться вершиной соответствующего полиэдра Клейна, треугольники $\vec v\vec e_1\vec e_2$, $\vec v\vec e_2\vec e_3$, $\vec v\vec e_3\vec e_1$ будут оставаться его гранями и, стало быть, будет сохраняться рёберная звезда $\starv$. При этом ввиду \eqref{eq:Pi_is_large} и \eqref{eq:det_is_huge} $\Pi(\vec w)$ будет отделено от нуля, в то время как $\det\starw$ может быть сколь угодно большим.

Итак, буквальное обращение предложения \ref{prop:starw_geq_cPi} невозможно. Тем не менее, уместно отметить, что в \cite{german_2007} доказано, что $\Pi(\vec x)$ отделено от нуля на ненулевых точках решётки $\La$ тогда и только тогда, когда ограничены определители рёберных звёзд и граней полиэдра Клейна решётки $\La$, соответствующего положительному ортанту. Чтобы получить такой результат, необходимо рассматривать вершины как самого полиэдра $\cK$, так и смежных с ним --- соответствующих другим ортантам. Так мы приходим к двум вопросам, ответы на которые позволили бы обратить неравенство \eqref{eq:omega_vs_det_starv} (изменив, возможно, константу).

\emph{Вопрос 1:} Верно ли, что если $\det\starv$ велик, то <<вблизи>> $\vec v$ найдётся вершина $\vec v'$ одного из $8$ полиэдров Клейна решётки $\La$ с малым $\Pi(\vec v')$?

\emph{Вопрос 2:} Существует ли отличный от $\det\starv$ целочисленный аффинный инвариант рёберной звезды, для которого справедлив <<обратимый>> аналог предложения \ref{prop:starw_geq_cPi}?

\paragraph{Благодарности.} Второй автор является победителем конкурса <<Junior Leader>> Фонда развития теоретической физики и математики <<БАЗИС>> и хотел бы поблагодарить жюри и спонсоров конкурса.

\end{document}